\documentclass{article}
\usepackage{amsmath,amsfonts,amssymb,amsthm,mathrsfs,bbm,stmaryrd,array}
\newtheorem{lemma}{Lemma}

\newtheorem{theorem}{Theorem}

\def\mydash{\CJKglue\raise0.2ex\hbox{---\kern-0.01em---}\CJKglue}

\begin{document}
\title{A congruence involving harmonic sums modulo $p^{\alpha}q^{\beta}$\footnotetext{\noindent This work is  supported by the National Natural Science Foundation of China, Project (No.10871169) and the Natural Science Foundation of Zhejiang Province, Project (No. LQ13A010012).}}
\author{\scshape Tianxin Cai$^{1\ast}$   Zhongyan Shen$^{2\ast\ast}$  Lirui Jia$^{1\ast\ast\ast}$\\
1 Department of Mathematics, Zhejiang University\\ Hangzhou 310027, P.R. China\\
2 Department of Mathematics, Zhejiang International Study University\\Hangzhou 310012, P.R. China\\
$\ast$txcai@zju.edu.cn\\
 $\ast\ast$huanchenszyan@163.com\\
$\ast\ast\ast$ jialirui@126.com}
\date{}
\maketitle
\textbf{Abstract} In 2014, Wang and Cai established the following
harmonic congruence for any odd prime $p$ and positive integer $r$,
\begin{equation*}
 Z(p^{r})\equiv-2p^{r-1}B_{p-3} ~(\bmod ~ p^{r}),
\end{equation*}
where $ Z(n)=\sum\limits_{i+j+k=n\atop{i,j,k\in\mathcal{P}_{n}}}\frac{1}{ijk}$ and $\mathcal{P}_{n}$ denote the set of positive integers which
are prime to $n$.\\
In this note, we  obtain a congruence for distinct odd primes $p,~q$ and positive integers $\alpha,~\beta$,
 \begin{equation*}
  Z(p^{\alpha}q^{\beta})\equiv 2(2-q)(1-\frac{1}{q^{3}})p^{\alpha-1}q^{\beta-1}B_{p-3}\pmod{p^{\alpha}}
  \end{equation*}
and the necessary and sufficient condition for
\begin{equation*}
  Z(p^{\alpha}q^{\beta})\equiv 0\pmod{p^{\alpha}q^{\beta}}.
  \end{equation*}
Finally, we raise a conjecture that for $n>1$ and odd prime power
$p^{\alpha}||n$, $\alpha\geq1$,
\begin{eqnarray}
    \nonumber Z(n)\equiv \prod\limits_{q|n\atop{q\neq p}}(1-\frac{2}{q})(1-\frac{1}{q^{3}})(-\frac{2n}{p})B_{p-3}\pmod{p^{\alpha}}.
  \end{eqnarray}


\section{Introduction.}
Let
\begin{equation*}
 Z(n)=\sum\limits_{i+j+k=n\atop{i,j,k\in\mathcal{P}_{n}}}\frac{1}{ijk},
\end{equation*}
where $\mathcal{P}_{n}$ denotes the set of positive integers which
are prime to $n$.\\

At the beginning of the 21th century, Zhao (Cf.\cite{Zh}) first
announced the following curious congruence involving multiple
harmonic sums for any odd prime $p>3$,
\begin{equation}
 Z(p)\equiv-2B_{p-3} ~(\bmod ~ p), \label{eq:1}
\end{equation}
which holds when $p=3$ evidently. Here, Bernoulli numbers $B_{k}$
are defined by the recursive relation:
\begin{center}
 $\sum\limits_{i=0}^{n}\binom{n+1}{i}B_{i}=0,n\geq 1 .$
\end{center}
A simple proof of  \eqref{eq:1} was presented in \cite{Ji}.
 This congruence has been generalized
along several directions. First, Zhou and Cai \cite{ZC} established
the following harmonic congruence for prime $p > 3$ and integer
$n\leq p-2$
\begin{equation*}
 \sum\limits_{l_{1}+l_{2}+\cdots+l_{n}=p}\frac{1}{l_{1}l_{2}\cdots l_{n}}\equiv\begin{cases}-(n-1)!B_{p-n} ~(\bmod ~ p)~~~ if~ 2\nmid n,\\
 -\frac{n(n!)}{2(n+1)}pB_{p-n-1} ~(\bmod  ~p^{2}) ~~~if~ 2\mid n.\\
 \end{cases}
\end{equation*}
Later, Xia and Cai \cite{XC} generalized \eqref{eq:1} to
\begin{equation*}
Z(p)\equiv-\frac{12B_{p-3}}{p-3}-\frac{3B_{2p-4}}{p-4} ~(\bmod ~ p^{2}),
\end{equation*}
where $p>5$ is a prime. \\
Recently, Wang and Cai \cite{WC} proved for every prime $p\geq 3$
and positive integer $r$,
\begin{equation}
Z(p^{r})\equiv-2p^{r-1}B_{p-3} ~(\bmod ~ p^{r}).\label{eq:13}
\end{equation}
Let $n=2$ or 4, for every positive integer $r\geq \frac{n}{2}$ and
prime $p> n$, Zhao \cite{Zh1} generalized \eqref{eq:13} to
\begin{equation*}
 \sum\limits_{i_{1}+i_{2}+\cdots+i_{n}=p^{r}\atop{i_{1},i_{2},\cdots,i_{n}\in \mathcal{P}_{p}}}\frac{1}{i_{1}i_{2}\cdots i_{n}}
 \equiv-\frac{n!}{n+1}p^{r}B_{p-n-1} ~(\bmod  ~p^{r+1}).
\end{equation*}

In this paper, we obtain the following theorems.

\begin{theorem}
Let $p,~q$ be distinct odd primes, then
 \begin{equation*}
  Z(pq)\equiv 2(2-q)(1-\frac{1}{q^{3}})B_{p-3}\pmod{p}.
  \end{equation*}
\end{theorem}

\begin{theorem}
Let $p,~q$ be distinct odd primes and $\alpha,~\beta$ positive integers, then
 \begin{equation*}
  Z(p^{\alpha}q^{\beta})\equiv 2(2-q)(1-\frac{1}{q^{3}})p^{\alpha-1}q^{\beta-1}B_{p-3}\pmod{p^{\alpha}}.
  \end{equation*}
\end{theorem}

\begin{theorem}
Let $p,~q$ be distinct odd primes and $\alpha,~\beta$ positive integers, if and only if $p=q^{2}+q+1$ or $q=p^{2}+p+1$ or $p|q^{2}+q+1$ and $q|p^{2}+p+1$, we have
 \begin{equation*}
  Z(p^{\alpha}q^{\beta})\equiv 0\pmod{p^{\alpha}q^{\beta}}.
  \end{equation*}
\end{theorem}

Finally, we have the following\\

\textbf{Conjecture}~ For any positive integer $n>1$ and odd prime
power $p^{\alpha}||n~(p^{\alpha}|n$, $p^{\alpha+1}\nmid n)$,
$\alpha\geq 1$, then
\begin{eqnarray}
    \nonumber Z(n)\equiv \prod\limits_{q|n\atop{q\neq p}}(1-\frac{2}{q})(1-\frac{1}{q^{3}})(-\frac{2n}{p})B_{p-3}\pmod{p^{\alpha}}.
  \end{eqnarray}
  This is the generalization of  Theorem 2 and \eqref{eq:13}.\\
\section{Preliminaries.}
In order to prove the theorems, we need the following lemmas.
\begin{lemma}[\cite{WC}]
 Let $p$ be odd prime and $r,~m$ positive integers, then
\begin{equation*}
 Z(p^{r})\equiv -2p^{r-1}B_{p-3}~\pmod{p^{r}},
\end{equation*}
\begin{equation*}
 \sum\limits_{i+j+k=mp^{r}\atop{i,j,k\in\mathcal{P}_{p}}}\frac{1}{ijk}\equiv mZ(p^{r})~\pmod{p^{r}}.
\end{equation*}
\end{lemma}

\begin{lemma} [\cite{M},\cite{Sl}]
 Let $p$ be odd prime and $l$ positive integer, then
  \begin{equation*}
  \sum\limits_{k=1\atop{(k,p)=1}}^{p^{l}-1}\frac{1}{k^{s}}\equiv \begin{cases}
  0~(\bmod~p^{2l-1}), ~for~odd~s~with~p-1|s+1~and~p\nmid s,\\
  0~(\bmod~p^{2l}), ~for~odd~s~with~p-1\nmid s+1~or~p|s,\\
  0~(\bmod~p^{l-1}), ~for~even~s~with~p-1|s,\\
  0~(\bmod~p^{l}), ~for~even~s~with~p-1\nmid s.\\
  \end{cases}
  \end{equation*}
\end{lemma}

Define
 \begin{eqnarray} \nonumber
S(n;p)=\sum_{a=1\atop{(a,p)=1}}^{n-1}\frac{1}{a^{2}}\sum\limits_{i=1\atop{(i,p)=1}}^{am-1}\frac{1}{i},
~T(n;p)=\sum_{a=1\atop{(a,p)=1}}^{n-1}\frac{1}{a^{2}}\sum\limits_{i=1\atop{i\equiv
am\pmod{p}}}^{am-1}\frac{1}{i}.
  \end{eqnarray}
\begin{lemma}
 Let $p$ be odd prime and $m$ positive integer coprime to $p$, then
\begin{eqnarray}
\nonumber S(p;p)\equiv m^{2}B_{p-3}\pmod{p}.
  \end{eqnarray}
\end{lemma}
\begin{proof}
When $(i,p)=1$, then $\frac{1}{i}\equiv i^{p-2}\pmod{p}$ by
Euler's Theorem.
 For any positive integers $n$ and $r$, it is well-known that
  \begin{eqnarray}\label{1}
 \sum_{a=1}^{n-1} a^{r}=\frac{1}{r+1}\sum\limits_{k=0}^{r}\binom{r+1}{k}B_{k}n^{r+1-k},
  \end{eqnarray}
 hence
 \begin{eqnarray}\label{2}
\nonumber \sum_{a=1}^{p-1}\frac{1}{a^{2}}\sum\limits_{i=1\atop{(i,p)=1}}^{am-1}\frac{1}{i}
&\equiv& \sum_{a=1}^{p-1}\frac{1}{a^{2}}\sum\limits_{i=1\atop{(i,p)=1}}^{am-1}i^{p-2}
\equiv \sum_{a=1}^{p-1}\frac{1}{a^{2}}\sum\limits_{i=1}^{am-1}i^{p-2}\\
\nonumber &\equiv& \sum_{a=1}^{p-1}\frac{1}{a^{2}}\frac{1}{p-1}\sum\limits_{k=0}^{p-2}\binom{p-1}{k}B_{k}(am)^{p-1-k}\\
 &\equiv& \frac{m^{2}}{p-1}\sum\limits_{k=0}^{p-2}\frac{1}{p-1}\binom{p-1}{k}B_{k}\sum_{a=1}^{p-1}(am)^{p-3-k}\pmod{p}.
  \end{eqnarray}
 Since $(am,p)=1$, by Lemma 2, if and only if $k=p-3$, $\sum_{a=1}^{p-1}(am)^{p-3-k}$ is not congruence to 0 modulo $p$.
It follows from (\ref{2}) that
 \begin{eqnarray}
\nonumber \sum_{a=1}^{p-1}\frac{1}{a^{2}}\sum\limits_{i=1\atop{(i,p)=1}}^{am-1}\frac{1}{i}
\equiv \frac{m^{2}}{p-1}\binom{p-1}{p-3}B_{p-3}(p-1)\equiv m^{2}B_{p-3}\pmod{p}.
  \end{eqnarray}
This completes the proof of Lemma 3.
\end{proof}
\begin{lemma}
 Let $p$ be odd prime, $m$ positive integer coprime to $p$ and $\alpha\geq 2$ positive integer, then
\begin{eqnarray}
\nonumber S(p^{\alpha};p)\equiv p^{\alpha-1} m^{2}B_{p-3}\pmod{p^{\alpha}}.
  \end{eqnarray}
\end{lemma}
\begin{proof}
Let $a=s+p^{\alpha-1}t,~1\leq s\leq p^{\alpha-1}-1,~0\leq t\leq p-1,~(s,p)=1$, then
 \begin{eqnarray}
\nonumber S(p^{\alpha};p)&=&\sum_{t=0}^{p-1}\sum_{s=1\atop{(s,p)=1}}^{p^{\alpha-1}-1}\frac{1}{(s+p^{\alpha-1}t)^{2}}\sum\limits_{i=1\atop{(i,p)=1}}^{(s+p^{\alpha-1}t)m-1}\frac{1}{i}\\
\nonumber&\equiv& \sum_{t=0}^{p-1}\sum_{s=1\atop{(s,p)=1}}^{p^{\alpha-1}-1}\frac{1}{s^{2}}(1-\frac{2p^{\alpha-1}t}{s})\sum\limits_{i=1\atop{(i,p)=1}}^{sm-1}\frac{1}{i}\\
\nonumber&+&
\sum_{t=0}^{p-1}\sum_{s=1\atop{(s,p)=1}}^{p^{\alpha-1}-1}\frac{1}{s^{2}}(1-\frac{2p^{\alpha-1}t}{s})\sum\limits_{i=sm\atop{(i,p)=1}}^{sm+p^{\alpha-1}tm-1}\frac{1}{i}\pmod{p^{\alpha}}.
  \end{eqnarray}
It is easy to see that
 \begin{equation*}
2p^{\alpha-1}\sum_{t=0}^{p-1}t\sum_{s=1\atop{(s,p)=1}}^{p^{\alpha-1}-1}\frac{1}{s^{3}}\sum\limits_{i=1\atop{(i,p)=1}}^{sm-1}\frac{1}{i}\equiv 0\pmod{p^{\alpha}}.
  \end{equation*}
By Lemma 2, we have
 \begin{equation*}
\sum_{t=0}^{p-1}\sum_{s=1\atop{(s,p)=1}}^{p^{\alpha-1}-1}\frac{1}{s^{2}}(1-\frac{2p^{\alpha-1}t}{s})\sum\limits_{i=sm\atop{(i,p)=1}}^{sm+p^{\alpha-1}tm-1}\frac{1}{i}\equiv 0\pmod{p^{\alpha}}.
  \end{equation*}
  Therefore
  \begin{equation*}
S(p^{\alpha};p)\equiv \sum_{t=0}^{p-1}\sum_{s=1\atop{(s,p)=1}}^{p^{\alpha-1}-1}\frac{1}{s^{2}}\sum\limits_{i=1\atop{(i,p)=1}}^{sm-1}\frac{1}{i}\equiv pS(p^{\alpha-1};p)\equiv \cdots \equiv p^{\alpha-1}S(p;p)\pmod{p^{\alpha}}.
  \end{equation*}
  By Lemma 3,  we complete the proof of Lemma 4.
  \end{proof}
\begin{lemma}
Let $p,~q$ be distinct odd primes, $m$ positive integer coprime to
$p$ and $\alpha\geq 2,~\beta\geq 0$ integers, then
\begin{eqnarray}
\nonumber S(p^{\alpha}q^{\beta}£»p)\equiv p^{\alpha-1} q^{\beta}m^{2}B_{p-3}\pmod{p^{\alpha}}.
  \end{eqnarray}
\end{lemma}
\begin{proof}
Let $a=s+p^{\alpha}t,~1\leq s\leq p^{\alpha}-1,~0\leq t\leq q^{\beta}-1,~(s,p)=1$, then
 \begin{eqnarray}
\nonumber S(p^{\alpha}q^{\beta};p)&=&\sum_{t=0}^{q^{\beta}-1}\sum_{s=1\atop{(s,p)=1}}^{p^{\alpha}-1}\frac{1}{(s+p^{\alpha}t)^{2}}\sum\limits_{i=1\atop{(i,p)=1}}^{(s+p^{\alpha}t)m-1}\frac{1}{i}\\
\nonumber &\equiv& \sum_{t=0}^{q^{\beta}-1}\sum_{s=1\atop{(s,p)=1}}^{p^{\alpha}-1}\frac{1}{s^{2}}\sum\limits_{i=1\atop{(i,p)=1}}^{sm-1}\frac{1}{i}\\
\nonumber &+&\sum_{t=0}^{q^{\beta}-1}\sum_{s=1\atop{(s,p)=1}}^{p^{\alpha}-1}\frac{1}{s^{2}}\sum\limits_{i=sm\atop{(i,p)=1}}^{sm+p^{\alpha}tm-1}\frac{1}{i}\pmod{p^{\alpha}}.
  \end{eqnarray}
By Lemma 2, we have
 \begin{equation*}
\sum\limits_{i=sm\atop{(i,p)=1}}^{sm+p^{\alpha}tm-1}\frac{1}{i}\equiv
0 \pmod{p^{\alpha}}.
  \end{equation*}
  Therefore
  \begin{equation*}
S(p^{\alpha}q^{\beta};p)\equiv q^{\beta}S(p^{\alpha})\pmod{p^{\alpha}}.
  \end{equation*}
  By Lemma 4,  we complete the proof of Lemma 5.
\end{proof}
\begin{lemma}[\cite{S}]
 Let $p$ be odd prime, $m\in{Z}^{+},(m,p)=1$, $[x]$ denote the largest integer less than or equal to $x$, then
\begin{equation*}
  \sum_{a=1}^{p-1}\frac{1}{a^k}\Big[\frac{am}{p}\Big]\equiv\left\{\begin{array}
   {l@{,\quad}l}
   -\frac{m+1}{2}\pmod{p}&k=0,\\
   0\pmod{p}&1\leq k \leq p-2,k\text{~is ~even}, \\
   \frac{m^k-m^p}{k}B_{p-k}\pmod{p}&1\leq k \leq p-2,k\text{~is ~odd}. \\
   \end{array}\right.
\end{equation*}
\end{lemma}

Meanwhile, it is easy to see that
 \begin{eqnarray}\label{3}
T(p;p)\equiv\frac{1}{m}\sum_{a=1}^{p-1}\frac{1}{a^3}\Big[\frac{am}{p}\Big]\equiv\frac{m^3-m}{3m}B_{p-3}\pmod{p}.
  \end{eqnarray}

\begin{lemma}
 Let $p$ be odd prime, $m$ positive integer coprime to $p$ and $\alpha\geq 2$ integer, then
\begin{eqnarray}
\nonumber T(p^{\alpha};p)\equiv \frac{m^{3}-m}{3m}p^{\alpha-1} B_{p-3}\pmod{p^{\alpha}}.
  \end{eqnarray}
\end{lemma}
\begin{proof}
Let $a=s+p^{\alpha-1}t,~1\leq s\leq p^{\alpha-1}-1,~0\leq t\leq p-1,~(s,p)=1$, then
  \begin{eqnarray}
\nonumber T(p^{\alpha};p)&=&\sum_{a=1\atop{(a,p)=1}}^{p^{\alpha}-1}\frac{1}{a^{2}}\sum\limits_{i=1\atop{i\equiv am\pmod{p}}}^{am-1}\frac{1}{i}
=\sum_{t=0}^{p-1}\sum_{s=1\atop{(s,p)=1}}^{p^{\alpha-1}-1}\frac{1}{(s+p^{\alpha-1}t)^{2}}\sum\limits_{i=1\atop{i\equiv (s+p^{\alpha-1}t)m\pmod{p}}}^{(s+p^{\alpha-1}t)m-1}\frac{1}{i}\\
\nonumber &\equiv &\sum_{t=0}^{p-1}\sum_{s=1\atop{(s,p)=1}}^{p^{\alpha-1}-1}\frac{1}{s^{2}}(1-\frac{2p^{\alpha-1}t}{s})(\sum\limits_{i=1\atop{i\equiv sm\pmod{p}}}^{sm-1}\frac{1}{i} +\sum\limits_{i=sm\atop{i\equiv sm\pmod{p}}}^{sm+p^{\alpha-1}tm-1}\frac{1}{i})\pmod{p^{\alpha}}.
\end{eqnarray}
It is easy to see that
 \begin{equation*}
2p^{\alpha-1}\sum_{t=0}^{p-1}t\sum_{s=1\atop{(s,p)=1}}^{p^{\alpha-1}-1}\frac{1}{s^{3}}\sum\limits_{i=1\atop{i\equiv sm\pmod{p}}}^{sm-1}\frac{1}{i}\equiv 0\pmod{p^{\alpha}}.
  \end{equation*}
Since
 \begin{eqnarray}
\nonumber\sum\limits_{i=sm\atop{i\equiv sm\pmod{p}}}^{sm+p^{\alpha-1}tm-1}\frac{1}{i}&\equiv& \sum\limits_{j=0}^{p^{\alpha-2}tm-1}\frac{1}{sm+jp}\equiv \sum\limits_{j=0}^{p^{\alpha-2}tm-1}\frac{1}{sm(1+\frac{jp}{sm})}\equiv \sum\limits_{j=0}^{p^{\alpha-2}tm-1}\frac{1}{sm}\sum\limits_{k=0}^{\alpha-1}(-\frac{jp}{sm})^{k}\\
\nonumber&\equiv& \sum\limits_{k=0}^{\alpha-1}\frac{(-p)^{k}}{(sm)^{k+1}}\sum\limits_{j=0}^{p^{\alpha-2}tm-1}j^{k},
  \end{eqnarray}
  by (\ref{1}), we have $\sum\limits_{j=0}^{p^{\alpha-2}tm-1}j^{k}\equiv 0\pmod{p^{\alpha-2}}$. Together with Lemma 2,
  we have
 \begin{equation*}
  \sum_{t=0}^{p-1}\sum_{s=1\atop{(s,p)=1}}^{p^{\alpha-1}-1}\frac{1}{s^{2}}(1-\frac{2p^{\alpha-1}t}{s})\sum\limits_{i=sm\atop{i\equiv sm\pmod{p}}}^{sm+p^{\alpha-1}tm-1}\frac{1}{i}\equiv 0\pmod{p^{\alpha}}.
  \end{equation*}
Hence
 \begin{eqnarray}
\nonumber T(p^{\alpha};p)\equiv \sum_{t=0}^{p-1}\sum_{s=1\atop{(s,p)=1}}^{p^{\alpha-1}-1}\frac{1}{s^{2}}\sum\limits_{i=1\atop{i\equiv sm\pmod{p}}}^{sm-1}\frac{1}{i}
\equiv p T(p^{\alpha-1};p)\equiv \cdots\equiv p^{\alpha-1} T(p;p)\pmod{p^{\alpha}},
\end{eqnarray}
 By (\ref{3}),  we complete the proof of Lemma 7.
\end{proof}
\begin{lemma}
Let $p,~q$ be distinct odd primes, $m$ positive integer coprime to
$p$ and $\alpha\geq 2,~\beta\geq 0$ integers, then
\begin{eqnarray}
\nonumber T(p^{\alpha}q^{\beta};p)\equiv \frac{m^{3}-m}{3m}p^{\alpha-1}q^{\beta} B_{p-3}\pmod{p^{\alpha}}.
  \end{eqnarray}
\end{lemma}
\begin{proof}
Let $a=s+p^{\alpha}t,~1\leq s\leq p^{\alpha}-1,~0\leq t\leq q^{\beta-1}-1,~(s,p)=1$, then
 \begin{eqnarray}
\nonumber T(p^{\alpha}q^{\beta};p)&=&\sum_{t=0}^{q^{\beta}-1}\sum_{s=1\atop{(s,p)=1}}^{p^{\alpha}-1}\frac{1}{(s+p^{\alpha}t)^{2}}\sum\limits_{i=1\atop{i\equiv (s+p^{\alpha}t)m\pmod{p}}}^{(s+p^{\alpha}t)m-1}\frac{1}{i}\\
\nonumber &\equiv& \sum_{t=0}^{q^{\beta}-1}\sum_{s=1\atop{(s,p)=1}}^{p^{\alpha}-1}\frac{1}{s^{2}}\sum\limits_{i=1\atop{i\equiv sm\pmod{p}}}^{sm-1}\frac{1}{i}
+\sum_{t=0}^{q^{\beta}-1}\sum_{s=1\atop{(s,p)=1}}^{p^{\alpha}-1}\frac{1}{s^{2}}\sum\limits_{i=sm\atop{i\equiv sm\pmod{p}}}^{sm+p^{\alpha}tm-1}\frac{1}{i}\pmod{p^{\alpha}}.
  \end{eqnarray}
Since
 \begin{equation*}
\sum\limits_{i=sm\atop{i\equiv sm\pmod{p}}}^{sm+p^{\alpha}tm-1}\frac{1}{i}\equiv \sum\limits_{j=0}^{p^{\alpha-1}tm-1}\frac{1}{sm+jp}\pmod{p^{\alpha}},
  \end{equation*}
similar to Lemma 6, we can prove
\begin{eqnarray}
\nonumber \sum_{t=0}^{q^{\beta}-1}\sum_{s=1\atop{(s,p)=1}}^{p^{\alpha}-1}\frac{1}{s^{2}}\sum\limits_{i=sm\atop{i\equiv sm\pmod{p}}}^{sm+p^{\alpha}tm-1}\frac{1}{i}\equiv 0\pmod{p^{\alpha}}.
\end{eqnarray}

 Therefore
  \begin{equation*}
T(p^{\alpha}q^{\beta};p)\equiv \sum_{t=0}^{q^{\beta}-1}\sum_{s=1\atop{(s,p)=1}}^{p^{\alpha}-1}\frac{1}{s^{2}}\sum\limits_{i=1\atop{i\equiv sm\pmod{p}}}^{sm-1}\frac{1}{i}\equiv q^{\beta}T(p^{\alpha};p)\pmod{p^{\alpha}}.
  \end{equation*}
  By Lemma 7,  we complete the proof of Lemma 8.
\end{proof}

\section{Proofs of the Theorems.}
\begin{proof}[Proof of Theorem 1 \\]
By symmetry, It is easy to see that
 \begin{eqnarray}\label{18}
      Z(pq)=\sum\limits_{i+j+k=pq\atop{i,j,k\in\mathcal{P}_{pq}}}\frac{1}{ijk}
   =\sum\limits_{i+j+k=pq\atop{i,j,k\in\mathcal{P}_{p}}}\frac{1}{ijk}-3\sum_{a=1}^{p-1}\frac{1}{(p-a)q}
    \sum\limits_{i+j=aq\atop{(ij,p)=1}}\frac{1}{ij}+2\sum\limits_{a+b+c=p\atop{a,b,c\in\mathcal{P}_{p}}}\frac{1}{aqbqcq}.
  \end{eqnarray}
By Lemma 1,
 \begin{eqnarray}\label{19}
     \sum\limits_{i+j+k=pq\atop{i,j,k\in\mathcal{P}_{p}}}\frac{1}{ijk}
     +2\sum\limits_{a+b+c=p\atop{a,b,c\in\mathcal{P}_{p}}}\frac{1}{aqbqcq}\equiv(q+\frac{2}{q^{3}})Z(p)\equiv-2(q+\frac{2}{q^{3}})B_{p-3}\pmod{p}.
  \end{eqnarray}
Again by symmetry, the second sum in (\ref{18}) equals to
\begin{eqnarray}\label{20}
\nonumber &&-3\sum_{a=1}^{p-1}\frac{1}{(p-a)q}\frac{1}{aq}\sum\limits_{i+j=aq\atop{(ij,p)=1}}\frac{i+j}{ij}\\
    \nonumber &\equiv&\frac{6}{q^{2}}\sum_{a=1}^{p-1}\frac{1}{a^{2}}\sum\limits_{i=1\atop{(i,p)=(aq-i,p)=1}}^{aq}\frac{1}{i}\\
     \nonumber &\equiv&\frac{6}{q^{2}}\sum_{a=1}^{p-1}\frac{1}{a^{2}}\sum\limits_{i=1\atop{(i,p)=1}}^{aq-1}\frac{1}{i}
     -\frac{6}{q^{2}}\sum_{a=1}^{p-1}\frac{1}{a^{2}}\frac{1}{aq}[\frac{aq}{p}]\\
      &\equiv&\frac{6}{q^{2}}\sum_{a=1}^{p-1}\frac{1}{a^{2}}\sum\limits_{i=1\atop{(i,p)=1}}^{aq-1}\frac{1}{i}
     -\frac{6}{q^{3}}\sum_{a=1}^{p-1}\frac{1}{a^{3}}[\frac{aq}{p}]\pmod{p}.
  \end{eqnarray}
Applying Lemma 3 and Lemma 6 to (\ref{20}), we have
  \begin{eqnarray}\label{21}
-3\sum_{a=1}^{p-1}\frac{1}{(p-a)q}\sum\limits_{i+j=aq\atop{(ij,p)=1}}\frac{1}{ij}
    \equiv\frac{6}{q^{2}}q^{2}B_{p-3} -\frac{6}{q^{3}}\frac{q^{3}-q}{3}B_{p-3}
    \equiv 2(1+\frac{1}{q^{2}})B_{p-3}\pmod{p}.
  \end{eqnarray}
 Combining (\ref{18}), (\ref{19}) and (\ref{21}), we have
  \begin{eqnarray}
  \nonumber    Z(pq)\equiv-2(q+\frac{2}{q^{3}})B_{p-3}+2(1+\frac{1}{q^{2}})B_{p-3}\equiv 2(2-q)(1-\frac{1}{q^{3}})B_{p-3}\pmod{p}.
  \end{eqnarray}
 This completes the proof of Theorem 1.
\end{proof}

\textbf{Remark 1}~  When $n=pq$, $p,~q$ are distinct odd primes,
\begin{eqnarray}
    \nonumber Z(n)\equiv 2(2-q)(1-\frac{1}{q^{3}})B_{p-3}\equiv 6(1+\frac{3}{\phi(n)-2})(1+\frac{1}{(\phi(n)-1)^{3}})B_{\phi(n)-2}\pmod{p},
  \end{eqnarray}
where Euler function $\phi(n)=(p-1)(q-1)$, and we use Kummer congruence,
 \begin{eqnarray}
    \nonumber \frac{B_{\phi(n)-2}}{\phi(n)-2}\equiv\frac{B_{p-3}}{p-3}\pmod{p}.
  \end{eqnarray}
 Similarly,
 \begin{eqnarray}
    \nonumber Z(n)\equiv 6(1+\frac{3}{\phi(n)-2})(1+\frac{1}{(\phi(n)-1)^{3}})B_{\phi(n)-2}\pmod{q},
  \end{eqnarray}
  Therefore, by Chinese Remainder Theorem, we have
 \begin{eqnarray}
    \nonumber Z(n)\equiv 6(1+\frac{3}{\phi(n)-2})(1+\frac{1}{(\phi(n)-1)^{3}})B_{\phi(n)-2}\pmod{n}.
  \end{eqnarray}
\begin{proof}[Proof of Theorem 2 \\]
If $\alpha=1,~\beta=1$, Theorem 2 is Theorem 1. Without loss of
generality, suppose $\alpha\geq 2,~\beta\geq 1$, similar to the
proof of Theorem 1, we have
 \begin{eqnarray}\label{22}
  \nonumber    Z(p^{\alpha}q^{\beta})&=&\sum\limits_{i+j+k=p^{\alpha}q^{\beta}\atop{i,j,k\in\mathcal{P}_{pq}}}\frac{1}{ijk}\\
   \nonumber &=&\sum\limits_{i+j+k=p^{\alpha}q^{\beta}\atop{i,j,k\in\mathcal{P}_{p}}}\frac{1}{ijk}-3\sum_{a=1\atop{(a,p)=1}}^{p^{\alpha}q^{\beta-1}-1}\frac{1}{(p^{\alpha}q^{\beta}-a)q}
    \sum\limits_{i+j=aq\atop{(ij,p)=1}}\frac{1}{ij}+2\sum\limits_{a+b+c=p^{\alpha}q^{\beta-1}\atop{a,b,c\in\mathcal{P}_{p}}}\frac{1}{aqbqcq}\\
    \nonumber &\equiv & \sum\limits_{i+j+k=p^{\alpha}q^{\beta}\atop{i,j,k\in\mathcal{P}_{p}}}\frac{1}{ijk}
   +\frac{2}{q^{3}}\sum\limits_{a+b+c=p^{\alpha}q^{\beta-1}\atop{a,b,c\in\mathcal{P}_{p}}}\frac{1}{abc}
   +\frac{6}{q^{2}}\sum_{a=1\atop{(a,p)=1}}^{p^{\alpha}q^{\beta-1}-1}\frac{1}{a^{2}}\sum_{i=1\atop{(i,p)=1}}^{aq-1}\frac{1}{i}\\
   && -\frac{6}{q^{2}}\sum_{a=1\atop{(a,p)=1}}^{p^{\alpha}q^{\beta-1}-1}\frac{1}{a^{2}}\sum_{i=1\atop{i\equiv aq \pmod{p}}}^{aq-1}\frac{1}{i}\pmod{p^{\alpha}}.
  \end{eqnarray}
By Lemma 1, Lemma 5 and  Lemma 8,  (\ref{22}) is congruent to
 \begin{eqnarray}
    \nonumber &&(q^{\beta}+\frac{2q^{\beta-1}}{q^{3}})(-2p^{\alpha-1})B_{p-3}+\frac{6}{q^{2}}p^{\alpha-1}q^{\beta-1}q^{2}B_{p-3}
    -\frac{6}{q^{2}}\frac{q^{3}-q}{3q}p^{\alpha-1}q^{\beta-1}B_{p-3}\\
   \nonumber &\equiv& 2(2-q)(1-\frac{1}{q^{3}})p^{\alpha-1}q^{\beta-1}B_{p-3}\pmod{p^{\alpha}}.
  \end{eqnarray}
 This completes the proof of Theorem 2.
\end{proof}

\begin{proof}[Proof of Theorem 3 \\]
 By Theorem 2, if and only if one of the following cases is true, $Z(p^{\alpha}q^{\beta})\equiv 0 \pmod{p^{\alpha}q^{\beta}}$.
\begin{align}
&(a)\begin{cases}
 p\equiv 2 \pmod{q}\\
 q\equiv 2 \pmod{p}\\
  \end{cases}
 ~~~~~~ (b)\begin{cases}
 p \equiv 2 \pmod{q}\\
 q^{3} \equiv 1 \pmod{p}\\
  \end{cases}
\nonumber\\&  (c)\begin{cases}
 p^{3} \equiv 1 \pmod{q}\\
 q  \equiv 2 \pmod{p}\\
  \end{cases}
  ~or~  (d)\begin{cases}
 p^{3} &\equiv 1 \pmod{q}\\
 q^{3} &\equiv 1 \pmod{p}\\
  \end{cases}
 \nonumber
\end{align}

It is obvious that there are no primes $p, ~q$ satisfying Case ($a$). \\

 For Case ($b$). Let $p=2+aq$, with $a$ odd. If $a=1$, $q^{3}\equiv -8 \equiv 1\pmod{p}$,
 then $p=3,~q=1$, $q$ is not a prime. Hence  $a\geq 3$. Since  $q^{3} \equiv 1 \pmod{p}$,
 then $q\equiv 1 \pmod{p}$ or $q^{2}+q+1 \equiv 0 \pmod{p}$. It is obvious that there are no primes $p,~q$
 satisfy $p \equiv 2 \pmod{q}$ and $q\equiv 1 \pmod{p}$. If  $q^{2}+q+1 \equiv 0 \pmod{p}$,
 let $q^{2}+q+1=bp$,  then $q^{2}+q+1 =2b +baq$. Hence $q|2b-1$, there exists a positive integer $c$ such that
 $2b-1=cq$. Therefore, $bp=\frac{(cq+1)(2+aq)}{2}=\frac{ac}{2}q^{2}+\frac{a+2c}{2}q+1>q^{2}+q+1$,
 impossible!\\

 Similarly, we can show that Case ($c$) is impossible too.\\

For Case ($d$).  It is obvious that $p \equiv 1 \pmod{q}$ and $
q\equiv 1 \pmod{p}$ are not possible. Similar to Case ($b$), it
can be justified that if $p \equiv 1 \pmod{q}$ and $
q^{2}+q+1\equiv 0 \pmod{p}$, then $q^{2}+q+1$ must be a prime.
It is true if we exchange $p$ with $q$.\\

This completes the proof of Theorem 3.
\end{proof}

\textbf{Remark 2}~ By Theorem 3, we know if and only if prime
pairs $(p,~q)$ satisfy $q=p^{2}+p+1$ or $p=q^{2}+q+1$ or
$p|q^{2}+q+1$ and $q|p^{2}+p+1$, then
$Z(p^{\alpha}q^{\beta})\equiv 0 \pmod{p^{\alpha}q^{\beta}}$. There
exist prime pairs $(p,~q)$ such that $q=p^{2}+p+1$, for example
$(p,~q)=(3,~13),(5,~31),(17,~307),(41,~1723),(59,~3541),(71,~5113),(89,~8011)$,etc..\\

\textbf{Problem 1}~Are there infinitely many prime pairs $(p,~q)$ such that $q=p^{2}+p+1$? \\

In 2004, Chao \cite{C} proposed that: Find all pairs of positive
integers $a$ and $b$ such that $a$ divides $b^{2}+b+1$ and $b$
divides $a^{2}+a+1$. There recurrence is: $a(1)=a(2)=1$ and
$a(n+1)=\frac{1+a(n)+a^2(n)}{a(n-1)}$ for $n>2$, then
\begin{equation*}
a(n) = (\frac{4}{3} - \frac{2\sqrt{21}}{7})(\frac{5 + \sqrt{21}}{2})^{n}+ (\frac{4}{3} + \frac{2\sqrt{21}}{7})(\frac{5 - \sqrt{21}}{2})^{n} + \frac{1}{3}.
\end{equation*}
 In fact, the pairs $(a(n),~a(n+1))$, $n>0$, are all the
 solutions\cite{h}:

 $a(n)$ for $n=1,\cdots, 28$  are: 1, 1, 3, 13, 61, 291, 1393, 6673, 31971, 153181, 733933, 3516483, 16848481, 80725921, 386781123, 1853179693, 8879117341, 42542407011, 203832917713, 976622181553, 4679277990051, 22419767768701, 107419560853453, 514678036498563, 2465970621639361, 11815175071698241, \\56609904736851843, 271234348612560973.\\

 We find three prime pairs $(p,~q)$ such that $p|q^{2}+q+1$ and $q|p^{2}+p+1$, i.e.,
  $(p,~q)=(3,~13),(13,~61),(22419767768701, ~107419560853453)$. What is the next pair?. \\

\textbf{Problem 2}~Are there infinitely many prime pairs $(p,~q)$
such that $p|q^{2}+q+1$ and $q|p^{2}+p+1$?\\

  \textbf{Remark 3}~By the conjecture and Chinese Remainder Theorem, we can count out
  the remainder of $Z(n)$ modulo $n$ for any positive integer
  $n$. However, we still have the problem as pointed out in \cite{WC}, i.e.\\

  \textbf{Problem 3}~Can we find an arithmetical function $f(n)$ such that
\begin{equation*}
 Z(n)\equiv f(n)\pmod{n}.
\end{equation*}

\textbf{Acknowledgements}~ The authors wish to thank Dr. Deyi Chen
for his kind help in calculation, especially in verifying the
conjecture.




\begin{thebibliography}{10}
\bibitem{C}\label{C} W. W. Chao, Problem 2981, Crux Mathematicorum, 30 (2004):430.
\bibitem{h}\label{h} https://oeis.org/A101368.
\bibitem{Ji}\label{Ji} C. G. Ji, A simple proof of a curious congruence by Zhao, Proc. Amer. Math. Soc., 133(2005): 3469-3472.
\bibitem{M}\label{M} R. Me$\breve{s}$trovi$\acute{c}$, Wolstenholme's theorem: its generalizations and extensions in the last hundred and fifty years(1862-2012), arxiv:1111.3057v2.
\bibitem{S}\label{S} $\breve{S}$tefan Porubsky, Further congruences invoving Bernoulli numbers, J. Number Theory 16(1983): 87-94.
\bibitem{Sl}\label{Sl} I. S. Slavutskii, Leudesdorf¡¯s theorem and Bernoulli numbers, Arch. Math. (Brno)
35 (1999): 299-303.
\bibitem{WC}\label{WC} L. Q. Wang, T. X. Cai , A curious congruence
modulo prime powers, J. Number Theory 144(2014): 15-24.
\bibitem{XC}\label{XC} B. Z. Xia, T. X. Cai , Bernoulli numbers and congruences for harmonic sums, Int. J. Number Theory, 6(2010):
849-855.
\bibitem{Zh1}\label{Zh1}   J. Q. Zhao, A super congruence involving multiple harmonic sums, arxiv:1404.3549.
\bibitem{Zh}\label{Zh}   J. Q. Zhao, Bernoulli Numbers, Wolstenholme¡¯s Theorem,
and $p^5$ Variations of Lucas¡¯ Theorem, J. Number Theory
123(2007): 18-26.
\bibitem{ZC}\label{ZC} X. Zhou, T. X. Cai , A Generalization of a Curious Congruence on
Harmonic Sums. Proc. Amer. Math. Soc., 135 (2007): 1329-1333.
\end{thebibliography}
\end{document}